\documentclass{article}
\usepackage{amsmath}
 \usepackage{amsfonts}
\usepackage{amsthm}
\usepackage{bbm}
\usepackage{hyperref}
\theoremstyle{plain}

\numberwithin{equation}{section}

\newtheorem{thm}{Theorem}[section]
\newtheorem{remark}[thm]{Remark}
\newtheorem{prop}[thm]{Proposition}

\newtheorem{lem}[thm]{Lemma}
\newtheorem{example}[thm]{Example}
\begin{document}
\noindent

\title{An exponential kernel associated with operators that have one-dimensional self-commutators}
\author{Kevin F. Clancey}
\date{August 2018}
\maketitle
\begin{abstract} The exponential kernel \[E{g}(\lambda,w) = \exp -\frac{1}{\pi}\int_{\mathbb{C} }
\frac{g(u)}{\overline{u-w} (u-\lambda) }  da(u ),\] where the compactly supported bounded measurable function $g$ satisfies $0 \leq g\leq 1,$ and suitably defined for all complex $\lambda, w,$ plays a role in the theory of Hilbert space operators with one-dimensional self-commutators and in the theory of quadrature domains. This article studies continuity and integral representation properties of $E_{g}$ with further applications of this exponential kernel to operators with one-dimensional self-commutator.
\end{abstract}

\section{Introduction}For $g$ a compactly supported bounded measurable function defined on the complex plane $\mathbb{C}$ that satisfies $0\leq g\leq1,$ let $E_{g}(\lambda,w)=E(\lambda,w)$ be defined by 
\begin{equation}\begin{gathered}\label{key} E(\lambda,w) = \exp -\frac{1}{\pi}\int_{\mathbb{C} }
\frac{u-w}{u-\lambda }
\frac{g(u)}{\vert u-w\vert^2} da(u )= \\ \exp -\frac{1}{\pi}\int_{\mathbb{C} }
\frac{g(u)}{\overline{u-w} (u-\lambda) }  da(u )\end{gathered}\end{equation}
for $\lambda\neq w,$ with $E(w,w)$ defined to be $0$ if $\frac{1}{\pi}\int_{\mathbb{C}}g(u)\vert u-w\vert^{-2} = \infty$ and equal to \[\exp-\frac{1}{\pi}\int_{\mathbb{C}} \frac{g(u)}{\vert u-w\vert^2} da(u )\] when \begin{equation}\label{finite}\frac{1}{\pi}\int_{\mathbb{C}} \frac{g(u)}{\vert u-w\vert^2} da(u )<\infty .\end{equation}  Here $a$ denotes area measure.  

The function $E_g$ first appeared in the study of bounded linear operators on a Hilbert space with one-dimensional self-commutator. To show this connection, let $T$ be a bounded linear operator on a Hilbert space $\mathcal H$ satisfying $T^{*}T-TT^{*}= \varphi\otimes\varphi.$ Throughout the following it will be assumed that $T$ is irreducible, which in this case is equivalent to the statement that there are no non-zero subspaces of $\mathcal H$ reducing $T$ where $T$ restricts to a normal operators. Pincus \cite{PNAS} established that there is a one-to-one correspondence between the unitary equivalence classes of the collection of such operators with the collection of equivalence classes of compactly supported Lebesgue measurable functions $g$ satisfying $0\leq g\leq 1.$ The (equivalence class of the) function $g_T$ associated with $T$ is called the principal function of $T.$ The principal function $g$ (the subscript $T$ will usually not be included on the principal function) first appeared in \cite{pincus1} in the study of spectral theory of self-adjoint singular integral operators on the real line. To continue the story of the connection of $E_{g_{_{T}}}$ with $T,$ we introduce the local resolvent. It develops that for $\lambda$ in $\mathbb{C}$ there is a unique solution of the equation $T_{\lambda}^{*}x=\varphi$ orthogonal to the kernel of $T_{\lambda}^{*} = (T-\lambda)^{*}.$ This solution is denoted $T_{\lambda}^{*-1}\varphi.$  The $\mathcal H$-valued function $T_{\lambda}^{*-1}\varphi$  defined for $\lambda \in \mathbb{C}$ was first investigated by Putnam \cite{Putnam} and Radjabalipour  \cite{Radj} and will be called the global-local resolvent associated with the operator $T.$ The following result from \cite{Clancey84} shows the connection between the function $E_{g_{T}}$ and the $\mathcal H$-valued function $T_{\lambda}^{*-1}\varphi$ .

\begin{thm} Let $T$ be an irreducible operator with one-dimensional self-commutator $T^{*}T-TT^{*}= \varphi\otimes\varphi.$  Let $g=g_{T}$ be the associated principal function and $T_{\lambda}^{*-1}\varphi, \lambda\in \mathbb{C}$  the associated global-local-resolvent. Then for $\lambda$ and $w$ in $\mathbb{C}$  \begin{equation} 1- (T_{w}^{*-1}\varphi,T_{\lambda}^{*-1}\varphi )= E_{g}(\lambda, w ) =  \exp-\frac{1}{\pi}{\int_{\mathbb{C}} \frac{u-w}{u-\lambda }\frac{g(u)}{\vert u-w\vert^2}} da(u ).\label{KFC}\end{equation}\end{thm}

For $\vert\lambda\vert$ and $\vert w\vert$ larger than $\Vert T\Vert,$ the result in the above theorem follows from early work on the principal function as presented in \cite{pincus}. The identity (\ref{KFC}) has consequences for both the operator $T$ and the function $E.$ For example, it follows easily from the weak continuity of the global-local resolvent and (\ref{KFC}) that the function $E$ is separately continuous. However, the fact that $E$ is  separately continuous without assuming the identity is far from obvious (see, \cite[p. 260]{putinarmartin}).  In descending chronological order, expository accounts of the relevant operator theory can be found in \cite{putinarmartin}, \cite{peller}, \cite{xia}, \cite{semiKFC}, and \cite{putnam1}. We also refer to these sources for historical accounts and many of the references to the area. 

In an unexpected direction, when the function $g$ is the characteristic function of a planar domain $\Omega,$ Putinar \cite{Putinar} made connections between the exponential kernel $E_{g},$ quadrature domains, and operators with one dimensional self-commutator. A recent account of these connections, including a discussion of properties of $E,$ and citations of earlier work can be found in the book \cite{bookgusput}. 

The main focus here is on the continuity and integral representation properties of $E_{g}.$ This is the content of the first part of this paper. We will study the continuity properties function $E_{g}$ without any reference to the associated operator $T.$ In the second part, we will offer some comments about the operator $T$ that can be gleaned from the identity (\ref{KFC}) and results in the first part.

\section{Continuity properties of E} In this section we will establish the sectional continuity of the function $E_{g}$ and show this function is locally Lipschitz at points of positive density of the measure $gda.$ A study of Cauchy transform representations of $E_{g}$ will also be presented. It should be remarked that the sectional continuity of $\vert E_{g}\vert$ was established in \cite{Clancey84} using methods similar to but less exact than those employed here. It will be assumed that $w$ is fixed and $E_g(\lambda, w)$ will be considered as a function of $\lambda.$  Unless stated otherwise, the function $g$ will be assumed to satisfy $0\leq g\leq1.$ 

\subsection{Sectional continuity} Fix the point $w.$ For $\lambda\neq w,$ the continuity of the function \begin{equation}f_{w}(\lambda) = -\frac{1}{\pi}\int_{\mathbb{C}}\frac{g(u)}{(\overline{u-w})(u-\lambda)}da(u) \label{logE}\end{equation} can be established by elementary means.  In particular, this follows from basic properties of the Cauchy transform that will be introduced below. Thus for $w$ fixed, a study of the sectional continuity of $E_{g}(\lambda, w)$ reduces to investigating continuity at $w.$ This will accomplished by first studying the continuity of (\ref{logE}) in the case $\frac{1}{\pi}\int_{\mathbb{C}}\frac{g(u)}{\vert u-w\vert^{2}}da(u) < \infty.$ 

We first derive some integral formulas for the case where the function $g$ in (\ref{logE}) is the characteristic function of well-chosen discs. 

The linear fractional mapping \[T_{w,\lambda}(u)=\frac{u-w}{u-\lambda}\] has invariant properties relative to the measure $\frac{1}{\vert u-w\vert ^2}da(u)$. This property can be used to compute the real and imaginary parts of the integral \[ \frac{-1}{\pi}\int_{D} \frac{u-w}{u - \lambda}\vert u-w\vert^{-2} da(u)\] over specific discs.

For $\alpha \in\bf R$ let $D_{\lambda,\alpha}$ be the disc with center $c_{\alpha}=\frac{w+\lambda +\alpha (\lambda - w)}{2}$ of radius $r_{\alpha}=\frac{\vert (\lambda-w) (1-\alpha)\vert}{2}$. For $\alpha < 1$,  \[D_{\lambda, \alpha}=\{u:Re\left[\frac{u-w}{u-\lambda} \right] < \frac{\alpha}{\alpha -1}\}\]
and for $\alpha >1$ \[D_{\lambda, \alpha}=\{u:Re\left[\frac{u-w}{u-\lambda} \right]  > \frac{\alpha}{\alpha -1}\}.\] We note that for $N>1$ \[D_{\lambda, \frac{N}{N+1}}=\{u:Re\left[\frac{u-w}{u-\lambda} \right] < -N\}\]
and \[D_{\lambda, \frac{N}{N-1}}=\{u:Re\left[\frac{u-w}{u-\lambda} \right]  > N\}.\]

For $0\leq\alpha$, a direct computation using a change of variables and polar coordinates shows \[-\frac{1}{\pi}\int_{D_{\lambda,\alpha}} Re\frac{u-w}{u-\lambda }\frac{1}{\vert u-w\vert^2} da(u ) = \frac{-1}{\pi}\int_{0}^{\frac{\pi}{2}} \log(\alpha^2\cos^2\theta + \sin^2\theta)d\theta =\ln\frac{2}{1+\alpha},\] and for $\alpha <0$

\begin{equation}\label{real}-\frac{1}{\pi}\int_{D_{\lambda,\alpha}} Re\frac{u-w}{u-\lambda }\frac{1}{\vert u-w\vert^2} da(u ) =\ln\frac{2}{1+\vert\alpha\vert}.\end{equation}

In a similar manner, for $\beta\neq 0$ in $\bf R,$ let $\Delta_{\lambda,\beta}$ be the disc of radius $r_{\beta} = \vert\frac{\beta (\lambda - w)}{2}\vert$ centered at $c_{\beta} =\lambda +\frac{i(\lambda-w)\beta}{2}.$

For $\beta >0,$\[\Delta_{\lambda,\beta}=\{u:Im\frac{u-w}{u-\lambda} < -\frac{1}{\beta}\},\] and for $\beta <0$

\[\Delta_{\lambda,\beta}=\{u:Im\frac{u-w}{u-\lambda} > -\frac{1}{\beta}\}.\] 

Another polar-coordinates computation shows that for $\beta\neq 0$

\begin{equation}\label{Imaginary}\begin{gathered}-\frac{1}{\pi}\int_{\Delta_{\lambda,\beta}} Im\frac{u-w}{u-\lambda }\frac{1}{\vert u-w\vert^2} da(u ) =\\ \frac{1}{\pi}\int_{0}^{\pi} \arctan (\beta +\cot\theta) d\theta =\arctan\{\frac{\beta}{2}\}.\end{gathered}\end{equation}

At the end of this paper it will be shown that the identities (\ref{real}) and (\ref{Imaginary}) are closely aligned and derivable from (\ref{KFC}) in the case where the operator $T$ is the unilateral shift. 

One consequence of the identities (\ref{real}) and (\ref{Imaginary}) is the following. Given $\varepsilon >0,$ there is an $M=M(\varepsilon)$ independent of $\lambda$ and $w$ such that for $N>M$ \begin{equation}\label{RE} \frac{1}{\pi}\int_{\left| Re\frac{u-w}{u-\lambda }\right| >N} \left| Re\frac{u-w}{u-\lambda }\right|\frac{1}{\vert u-w\vert^2} da(u ) < \varepsilon\end{equation}

and
\begin{equation}\label{IM}\frac{1}{\pi}\int_{\vert Im\frac{u-w}{u-\lambda }\vert>N} \left| Im\frac{u-w}{u-\lambda }\right|\frac{1}{\vert u-w\vert^2} da(u ) < \varepsilon .\end{equation} For technical reasons, it is required that $M>1.$

Using the above, one can directly establish the sectional continuity of the function $E.$ 
\begin{thm} \label{continuity} For $w$ fixed in $\mathbb{C}$ the function $E_{g}(\cdot, w)$ is continuous on $\mathbb{C}.$ \end{thm}
\begin{proof}
Case 1.   $\frac{1}{\pi}\int_{\mathbb{C}} \frac{g(u)}{\vert u-w\vert^2} da(u )<\infty .$ 

Let $\varepsilon >0.$ Let $N >M$ be fixed so that the inequalities (\ref{RE}) and (\ref{IM}) hold for all $\lambda\neq w.$ Choose $\delta_{0} >0$ such that for $0<\delta\leq\delta_{0}$ one has

 \begin{equation}\label{bound}\frac{1}{\pi}\int_{\vert u-w\vert <\delta} \frac{g(u)}{\vert u-w\vert^2} da(u )<\frac{\varepsilon}{N} .\end{equation} 
 There exists a $\delta_1< \delta_0$ such that for $\vert \lambda -w\vert <\delta_1,$ one has \[ \bigl\lvert\frac{u-w}{u-\lambda} -1\bigl\rvert < \varepsilon\  \text{for}\  \vert u-w\vert \geq \delta_0.\] 
 
 Note for $\lambda$ sufficiently close to $w,$ say for $\vert \lambda -w \vert <\delta_{2},$ the set \[U_N^{\lambda}=\{u:\vert Re\frac{u-w}{u-\lambda }\vert >N\} = D_{\lambda ,\frac{N}{N-1}}\cup D_{\lambda ,\frac{N}{N+1}}\] and the set \[V_N^{\lambda}= \{u:\vert Im\frac{u-w}{u-\lambda }\vert >N\} =\Delta_{\lambda, 1/ N} \cup\Delta_{\lambda, -1/ N}\] will be in the disc $\{ u: \vert u-w\vert <\delta_{1}\}.$ For $\vert \lambda -w\vert <\delta_2,$ we estimate \[D(\lambda ):=\left|\frac{1}{\pi}\int_{\mathbb{C}} \frac{u-w}{u-\lambda}\frac{g(u)}{\vert u-w\vert ^2} da(u) - \frac{1}{\pi}\int_{\mathbb{C}}\frac{g(u)}{\vert u-w\vert ^2} da(u)\right|\] separately over the sets 
 \[A = \{ u:\vert u-w\vert \geq \delta_1\}, B = \{u: \vert u-w\vert < \delta_1\}\backslash (U_N^{\lambda}\cup V_{N}^{\lambda}), \text{and}\ C= U_N^{\lambda}\cup V_{N}^{\lambda}.\] 

 \par On $A$ we have the estimate 
 
 \[\begin{aligned} \left|\frac{1}{\pi}\int_{A} \frac{u-w}{u-\lambda}\frac{g(u)}{\vert u-w\vert ^2} da(u) - \frac{1}{\pi}\int_{A}\frac{g(u)}{\vert u-w\vert ^2} da(u)\right| \leq \\ \frac{1}{\pi}\int_{A} \left|\frac{u-w}{u-\lambda}-1 \right|\frac{g(u)}{\vert u-w\vert ^2} da(u) \leq \frac{\varepsilon}{\pi} \int_{\mathbb{C}} \frac{g(u)}{\vert u-w\vert^2}da(u) . \end{aligned} \] 
 
 \par On $B$ \[\begin{gathered} \left|\frac{1}{\pi}\int_{B} \frac{u-w}{u-\lambda}\frac{g(u)}{\vert u-w\vert ^2} da(u) - \frac{1}{\pi}\int_{B}\frac{g(u)}{\vert u-w\vert ^2} da(u)\right| \leq \\ \frac{1}{\pi}\int_{B}\left|\frac{u-w}{u-\lambda}\right|\frac{g(u)}{\vert u-w\vert^2}da(u) + \frac{1}{\pi}\int_{B}\frac{g(u)}{\vert u-w\vert^2}da(u) \leq\\ \frac{1}{\pi}\int_{B}\left| Re \left[\frac{u-w}{u-\lambda}\right]\right|\frac{g(u)} {\vert u-w\vert^2}da(u) + \frac{1}{\pi}\int_{B}\left| Im \left[\frac{u-w}{u-\lambda}\right]\right|
 \frac{g(u)}{\vert u-w\vert^2} da(u) +\frac{\varepsilon}{N} \leq \\ 2N\frac{1}{\pi}\int_{B}\frac{g(u)}{\vert u-w\vert ^2} da(u) +\frac{\varepsilon}{N} \leq 3\varepsilon. \end{gathered}\]
 \par On $C$
 
\[\begin{gathered} \left|\frac{1}{\pi}\int_{C} \frac{u-w}{u-\lambda}\frac{g(u)}{\vert u-w\vert ^2} da(u) - \frac{1}{\pi}\int_{C}\frac{g(u)}{\vert u-w\vert ^2} da(u)\right| \leq \\ \frac{1}{\pi}\int_{C}\left| Re\left[\frac{u-w}{u-\lambda} \right] \right| \frac{g(u)}{\vert u-w\vert ^2} da(u)  + \frac{1}{\pi}\int_{C}\left| Im\left[\frac{u-w}{u-\lambda}\right]\right| \frac{g(u)}{\vert u-w\vert ^2} da(u)  + \frac{1}{\pi}\int_{C}\frac{g(u)}{\vert u-w\vert ^2} da(u).\end{gathered}\]  The last integral is less than the corresponding integral over $\{u:\vert u -w\vert < \delta_1\}$ and consequently less than $\varepsilon.$ The first of the two integrals on the right side of this last inequality can be estimated as follows:

\[\begin{gathered} \frac{1}{\pi}\int_{C}\left| Re\left[\frac{u-w}{u-\lambda} \right] \right| \frac{g(u)}{\vert u-w\vert ^2} da(u) = \\  \frac{1}{\pi}\int_{U_N^{\lambda}}\left| Re\left[\frac{u-w}{u-\lambda} \right] \right|\frac{g(u)}{\vert u-w\vert ^2} da(u) +  \frac{1}{\pi}\int_{V_N^{\lambda}\backslash U_N^{\lambda}}\left| Re\left[\frac{u-w}{u-\lambda} \right] \right| \frac{g(u)}{\vert u-w\vert ^2} da(u).\end{gathered}\] The first of these last two integrals is less than $\varepsilon$ and using the fact that $\left|Re\left[\frac{u-w}{u-\lambda} \right] \right|$ is less than $N$ off $U_N^{\lambda},$ it follows from equation (\ref{bound}) that the second integral is also less than $\varepsilon.$
 A similar argument shows \[\frac{1}{\pi}\int_{C}\left| Im\left[\frac{u-w}{u-\lambda}\right]\right| \frac{g(u)}{\vert u-w\vert ^2} da(u) < 2\varepsilon.\]  It follows from the above discussion that for $\vert \lambda -w\vert <\delta_{2}$ \[ D(\lambda) < \left( 7  + \frac{1}{\pi}\int_{\mathbb{C}} \frac{g(u)}{\vert u-w\vert^2} da(u)\right)\varepsilon .\] This completes the proof in Case 1.
 
 Case 2.   $\frac{1}{\pi}\int_{\mathbb{C}} \frac{g(u)}{\vert u-w\vert^2} da(u )=\infty .$ This case is easier then the first case and was established in \cite{Clancey84}. The result follows once it is shown that $\lim_{\lambda\to w} \vert E(\lambda, w)\vert = 0.$ For completeness, we include the details. The notation $D(w,r)$ will be used for the disc in $\mathbb{C}$ centered at $w$ of radius $r.$ By the monotone convergence theorem \[\lim_{\lambda\to w}\frac{1}{\pi}\int_{D(w, \vert\lambda - w\vert )} \frac{g(u)}{\vert u-w\vert^2} da(u) = \infty.\] In the notation introduced above, the disc $D(w, \vert\lambda - w\vert )$ coincides with $D_{\lambda, -1},$ which can be written as the disjoint union $(D_{\lambda, -1}\backslash D_{\lambda, 0})\cup D_{\lambda, 0}.$ Then \[\begin{gathered}\vert E(\lambda, w)\vert = \exp -\frac{1}{\pi}\int_{\mathbb{C}} Re\left[\frac{u-w}{u-\lambda }\right]\frac{g(u)}{\vert u-w\vert^2} da(u )\leq \\
\left(\exp -\frac{1}{\pi}\int_{\mathbb{C}\backslash D_{\lambda, -1}} Re\left[\frac{u-w}{u-\lambda}\right]\frac{g(u)}{\vert u-w\vert^2} da(u )\right)\left(\exp -\frac{1}{\pi}\int_{D_{\lambda, 0}} Re\left[\frac{u-w}{u-\lambda}\right]\frac{g(u)}{\vert u-w\vert^2} da(u )\right)\leq\\ 2 \exp -\frac{1}{2\pi}\int_{\mathbb{C}\backslash D( w,\vert \lambda -w\vert)} \frac{g(u)}{\vert u-w\vert^2} da(u ).\end{gathered}\] Here we used the facts that $Re\left[\frac{u-w}{u-\lambda}\right]$ is greater than $\frac{1}{2}$ on $\mathbb{C}\backslash D( w,\vert \lambda -w\vert)$ and non-negative on $D_{\lambda, -1}\backslash D_{\lambda, 0}$ as well as the result that \[-\frac{1}{\pi}\int_{D_{\lambda, 0}} Re\left[\frac{u-w}{u-\lambda}\right]\frac{1}{\vert u-w\vert^2} da(u) =\ln 2.\] As noted above, the desired result now follows from the monotone convergence theorem.

 \end{proof} \begin{remark}Assuming, as is the case here, that $0\leq g\leq 1,$ one consequence of the last integral inequality is the inequality \[\vert E_{g}(\lambda, w) \vert\leq 2, \ \text{for all}\ \lambda, w \]  with equality holding if and only if $g$ is the characteristic function of a disc and where $\lambda, w$ are antipodal boundary points.\end{remark}
 \begin{remark} \label{forrep} With $w$ fixed, with minor modifications, the proof of Case 1 in Theorem \ref{continuity} establishes the continuity of the integral in (\ref{logE}) as a function of $\lambda$ for any compactly supported bounded measurable function $g$  under the assumption \[\frac{1}{\pi}\int_{\mathbb{C}} \frac{|g(u)|}{\vert u-w\vert^2} da(u )<\infty .\]\end{remark}
 
\subsection{Local Lipschitz continuity} The function $E_{g}$ is Lipschitz at almost every point in the support of $g.$ To establish this we will use the following elementary lemmas.
 
 \begin{lem} Let $h=h(x)$ be continuous on the interval $[0,R]$ with $h(0)=0$ and $R>0$ is fixed.  For $t\in (0,R]$ define \[H(t) = \int_{t}^{R}\frac{h(x)}{x}dx.\] Given $0< \varepsilon $ there exists a $\delta >0$ and constant $K=K(\varepsilon)$ such that for $0<t<\delta$ one has the estimate \[\vert H(t)\vert\leq K -\varepsilon \ln t\]\end{lem}
 
 \begin{proof} Let $\varepsilon >0$ and find $\delta >0$ such that $0<t<\delta$ implies $\vert h(t)\vert <\varepsilon.$ Then for $0<t<\delta$ \[\begin{aligned}\vert H(t)\vert \leq \int_{t}^{\delta}\vert h(x)\vert\frac{dx}{x} + \int_{\delta}^{R}\vert h(x)\vert\frac{dx}{x}\leq\varepsilon\int_{t}^{\delta}\frac{dx}{x}+M\ln R -M\ln\delta =\\M\ln R +(\varepsilon-M)\ln \delta - \varepsilon\ln t ,\end{aligned}\] where $M$ is the maximum of $\vert h\vert$ on $[0,R].$ The result follows with $K(\varepsilon ) = M\ln R +(\varepsilon - M)\ln\delta.$ \end{proof}
 
We continue to assume that $g$ is a measurable function with compact support satisfying $0\leq g\leq 1$ and introduce the notation $\bold{L}_{g}$ for the set of points of positive Lebesgue density of $g.$ Thus $\bold{L}_{g}$  is the set of points $w$ that satisfy \[ \lim_{R\to 0}\frac{1}{\pi r^2} \int_{D(w, r)} g(u) da(u) = \gamma >0.\]
 
 \begin{lem}\label{Lip} Let $g=g(u)$ be a bounded non-negative measurable function defined in a neighborhood of $D(0,R)$ of where $0\in\bold{L}_{g},$ that is, \[\lim_{r\to 0}\frac{1}{\pi r^2}\int_{D(0,r)}g(u)da(u) =\gamma  >0. \] Given $\varepsilon >0$ there is a $\delta =\delta (\varepsilon )$ and a constant $K = K(\varepsilon )$ such that with $0<t<\delta$ one has the estimate \[\frac{1}{2\pi}\int_{D(0,R)\backslash D(0,t)}\frac{g(u)}{\vert u\vert^2}da(u) < K  - (\gamma -\varepsilon )\ln t.\] \end{lem}
\begin{proof} For $0\leq s\leq R,$ let  $G(s) = \frac{1}{2\pi}\int_{0}^{2\pi} g(se^{i\theta})d\theta$ and for $0<t\leq R$ set \[f(t)=\int_{t}^{R}\frac{G(s)}{s^2}sds.\] Note that \[f(t) =\frac{1}{2\pi}\int_{D(0,R)\backslash D(0,t)}\frac{g(u)}{\vert u\vert^2}da(u).\]Applying integration by parts for Lebesgue-Stieltjes integration, see, for example,  \cite{Integration}, to this first integral above for $f$  one obtains the identity
 \[\begin{gathered}f(t) = \frac{\int_{0}^{x}G(s)sds}{x^2}\bigg\rvert_{t}^{R} + \int_{t}^{R}\{\frac{2}{x^2}\int_{0}^{x}G(s)sds\}\frac{dx}{x}=\\\frac{1}{R^2}\int_{0}^{R}G(s)sds-\frac{1}{t^2}\int_{0}^{t}G(s)sds + \int_{t}^{R}\left[\frac{2}{x^2}\int_{0}^{x}G(s)sds-\gamma\right]\frac{dx}{x}+\gamma\ln R-\gamma\ln t\end{gathered}.\] We note that as $t\to 0$ the second term in this last expression approaches $\frac{1}{2}\gamma$ and as a consequence is bounded on $[0,R].$
 If one applies the preceding lemma to the third integral in the right side of this last identity with\[h(x)=\frac{2}{x^2}\int_{0}^{x}G(s)sds-\gamma\] the result follows.\end{proof}
  
 \begin{thm} \label{main} Let $w$ be in $\bold{L}_{g},$ that is, assume \[\lim_{\lambda\to w}\frac{1}{\pi\vert\lambda - w\vert ^2}\int_{D(w,\vert\lambda -w\vert)}g(u)da(u) =\gamma >0,\]  so that, $E(w,w)=0.$Then the function $E_{g}(\lambda , w)$ is Lipschitz at $w.$ More specifically, given $\varepsilon <\gamma$ there is a disc $D(w,\delta)$ with \begin{equation} |E(\lambda, w)|\leq K\vert\lambda -w\vert^{\gamma-\varepsilon}, \ \  \lambda\in D(w, \delta).\label{lipord}\end{equation}
 \end{thm}
 \begin{proof}
 First note that the disc $D(w,\vert\lambda -w\vert)$ is precisely $D_{\lambda, -1}$ and the complement of this disc is precisely the set where $Re\left[\frac{u-w}{u-\lambda}\right]\geq\frac{1}{2}.$  As a consequence \begin{equation}\begin{gathered}\label{prelim}\vert E(\lambda, w)\vert = \exp -\frac{1}{\pi}\int_{\mathbb{C}} Re\frac{u-w}{u-\lambda }\frac{g(u)}{\vert u-w\vert^2} da(u )\leq \\  2\exp -\frac{1}{2\pi}\int_{\mathbb{C}\backslash D(w,\vert\lambda - w\vert )} \frac{g(u)}{\vert u-w\vert^2} da(u ).  \end{gathered} \end{equation} This last identity is obtained by writing the integral \[-\frac{1}{\pi}\int_{D(w,\vert\lambda -w\vert)} Re\frac{u-w}{u-\lambda }\frac{g(u)}{\vert u-w\vert^2} da(u)\] over $D(w,\vert\lambda -w\vert) = D_{\lambda, -1}$  as the sum of the integrals over $D_{\lambda, -1}\backslash D_{\lambda, 0}$ and $D_{\lambda, 0}.$ The first of these integral being negative contributes nothing to the last inequality. In the second integral one can replace $g$ by $1$  and use the identity
 \[-\frac{1}{\pi}\int_{D_{\lambda, 0}} Re\frac{u-w}{u-\lambda }\frac{1}{\vert u-w\vert^2} da(u ) =\ln 2\]  to obtain the above bound. The result then follows from (\ref{prelim}) and Proposition  \ref{Lip}.
 \end{proof}

\begin{example} A simple example to keep in mind is the case where $g$ is the characteristic function $\mathbbm{1}_{\bf{D}} $ of the unit disc $\bf{D}.$
In this case, we denote the function $E_{g}$ by $E_{\bf{D}}.$ One computes \begin{equation} \label{unitdisc}E_{\bf{D}}(\lambda, w) = \begin{cases} \frac{\vert\lambda - w\vert^2}{1-\overline{w}\lambda} & \text{if}\ \lambda, w\in\bf{D} \\ \\
\overline{\left[\frac{w-\lambda}{w}\right]} & \text{if}\ \lambda\in\bf{D}\ \text{and}\ w\notin\bf{D}\\ \\1-\frac{1}{\overline{w}\lambda} &\text{if}\ \lambda, w\notin\bf{D} .\end{cases}\end{equation}
\end{example}
We remark that this example suggests that, in general, the local Lipschitz order of $E_{g}(w,\cdot)$ at $w$ in (\ref{lipord}) should be $2(\gamma - \varepsilon).$  This is the case if $g$ is smooth in a neighborhood of $w.$

\subsection {Cauchy Transform Representations of $E_{g}$} Given a compactly supported measure $\mu$ on $\mathbb{C},$ the Cauchy Transform $\hat{\mu}$ is the locally integrable function defined for $a.e.$  $\lambda$ in $\mathbb{C}$ by \begin{equation}\label{ct} \hat{\mu} (\lambda) = \frac{1}{\pi}\int_{\mathbb{C}}\frac{d\mu (u)}{u-\lambda} .\end{equation} For a measure of the form $d\mu = fda,$ where $f$ is a compactly supported integrable function, the Cauchy transform will be denote $\hat{f}.$ A good place to read about the Cauchy Transform is the monograph of Garnett \cite{Garnett}. 
In the sense of distributions \[-\bar{\partial}\hat{\mu} = \mu. \] 
Formally, for $w$ fixed and $\lambda\neq w,$ one expects in the sense of distributions \begin{equation}\label{dbarE} \bar{\partial_{\lambda}} E_{g}(\lambda, w) =\frac{E_{g}(\lambda, w)}{\overline{\lambda - w}}g(\lambda) \end{equation} and, consequently, \begin{equation}\label{distr} E_{g}(\lambda, w) =1-\left(\frac{E_{g}(\lambda, w)}{\overline{\lambda - w}}g(\lambda )\right)^{\widehat{ }}.\end{equation} Here,  $hat$ denotes the distributional Cauchy transform, given for a distribution $S$  with compact support on the test function $\phi$ by \[<\phi, \hat{S} > = -<\hat{\phi},S>.\]  Taking into account behavior at infinity and Weyl's Lemma, for $w$ fixed and $\lambda\neq w,$ formally, one further expects \begin{equation}\label{keyidentity} E(\lambda, w) = 1-\frac{1}{\pi}\int_{\mathbb{C}}\frac{E(u,w)}{\overline{u-w}} g(u)\frac{da(u)}{u-\lambda}.\end{equation} All of the above distributional identities are subtle. Although, for $w$ fixed, the function $\frac{E_{g}(\lambda, w)}{\overline{\lambda - w}}g(\lambda)$ is integrable, it is unclear whether the identity (\ref{keyidentity}) holds for all $\lambda.$   The goal of this subsection is to establish circumstances where, for $w$ fixed, (\ref{keyidentity}) holds for all $\lambda.$ More specifically, it will be shown that this is the case when $w\in\bold{L}_{g}$ or when $\frac{1}{\pi}\int_{\mathbb{C}} \frac{g(u)}{\vert u-w\vert^2} da(u )<\infty .$ We will take a somewhat ad hoc path to the results.

We will first consider an easy case.  Let $\bold{G}$ denote the essential support of the function $g.$ Thus  $z$ is in $\bold{G}$ if and only if every neighborhood of $z$ intersects the set $\{u : g(u)\neq 0\}$ in a set of positive measure. For the case $w\notin \bold{G}$ one can give a direct  proof that (\ref{keyidentity}) holds for all $\lambda.$ 
To this end, recall that for $h$ and $k$ bounded measurable functions with compact support in $\mathbb{C}$ one has \[ \hat{h}\hat{k} = \widehat{\hat{h}k}+\widehat{h\hat{k}}\] see, for example, \cite[p.107]{Garnett}. We remark that fact that $h$ and $k$ are bounded with compact support implies that $\hat{h}$ and $\hat{k}$ are continuous.   As a consequence, for $h$ bounded with compact support and $N\geq 1$ one has \begin{equation}\label{power}\hat{h}^{N} = N\widehat{\hat{h}^{N-1} h}.\end{equation}.

\begin{prop} Let $w\notin\bold{G}.$ Then (\ref{keyidentity}) holds for all $\lambda$ in $\mathbb{C}.$\end{prop}
\begin{proof} Applying the identity (\ref{power}) with $h(u)=\frac{g(u)}{\overline{u-w}}$  one sees
\[ \begin{gathered}1- E(\lambda, w) = \sum_{N=1}^{\infty}\frac{ (-1)^{N-1} }{N!} 
\left( \frac{1}{\pi}\int_{\mathbb{C}}\frac{g(u)}{\overline{u-w}} \frac{da(u)}{u-\lambda}\right)^N = \\   \sum_{N=1}^{\infty}\frac{ (-1)^{N-1} }{(N-1)!} \frac{1}{\pi}\int_{\mathbb{C}}
\left( \frac{1}{\pi}\int_{\mathbb{C}}\frac{g(v)}{\overline{v-w}} \frac{da(v)}{v-u}\right)^{N-1} \frac{g(u)}{\overline{u-w}}\frac{da(u)}{u-\lambda} =\\ \frac{1}{\pi}\int_{\mathbb{C}}\frac{E(u,w)}{\overline{u-w}}g(u)\frac{da(u)}{u-\lambda},\end{gathered}\] where, for $\lambda$ fixed, the interchange of summation and integration to produce the last expression follows from the uniform convergence of the series for $E_{g} (u, w)$ in the variable $u$ on $\bold{G},$ which is the support of the finite measure $\frac{g(u)}{\overline{u-w}}\frac{da(u)}{u-\lambda}.$
\end{proof}

\begin{remark} The extent of the validity of (\ref{power}) for arbitrary planar integrable $h$ is unclear; however, when $N=2$ it does hold $a.e.$ when $\hat{h}$ is integrable with respect to $\vert h\vert da$ (see, \cite{Volberg}).  \end{remark}

We continue our study of the validity of (\ref{dbar})  and (\ref{distr}) by first deriving and analogue of  (\ref{power}) for functions of the form \begin{equation}\label{singh} h_{w}(u)=\frac{g(u)}{\overline{u-w}}.\end{equation}  

Without loss of generality, it is sufficient to consider this last identity when $w=0.$ We begin by noting that for $\lambda\neq 0$ \begin{equation}\label{power1}\int_{\mathbb{C}}\frac{g(u)}{\overline{u}(u-\lambda)} da(u)= \frac{1}{\lambda}\left[\int_{\mathbb{C}}\frac{ug(u)}{\overline{u}(u-\lambda)}da(u) - \int_{\mathbb{C}}\frac{g(u)}{\overline{u}}da(u)\right].\end{equation} We introduce the notations\[ h_{0}(u)=\frac{g(u)}{\overline{u}}\ \ k_{0}(u)=\frac{ug(u)}{\overline{u}}\ \ C=-\frac{1}{\pi}\int_{\mathbb{C}}\frac{g(u)}{\overline{u}}da(u).\] Thus equation (\ref{power1}) can be written in the compact form \[ \widehat{h_{0}}(\lambda) =\frac{1}{\lambda}\left[ \widehat{k_{0}}(\lambda) + C\right]\ \lambda\neq 0.\]  Note that \[C=-\widehat{k_{0}}(0). \]Since the function $k_{0}$ is bounded with compact support, we can apply (\ref{power}) to this function. Using a binomial expansion we see for $\lambda\neq 0,$

\begin{equation}\label{binomial} \begin{gathered} \widehat{h_{0}}^{N}(\lambda) = \frac{1}{\lambda^N} 
\sum_{j=0}^{N}\frac{N!}{j!(N-j)!}\widehat{k_{0}}^{j}(\lambda)C^{N-j} =\\ 
\frac{C^{N}}{\lambda^{N}} + \frac{N}{\lambda^N}\sum_{j=1}^{N}\frac{(N-1)!}{(j-1)!(N-j)!} \widehat{(\widehat{k_{0}})^{j-1}k_{0}}(\lambda)C^{N-j}= \\
\frac{C^{N}}{\lambda^N} + \frac{N}{\lambda^N}\left( \left(\widehat{k_{0}+ C}\right)^{N-1}k_{0}\right) ^\bold{\widehat{ }}(\lambda).\end{gathered}\end{equation}

Therefore we have the following:

\begin{prop} Let $h_{0}$ be the function defined by (\ref{singh}) with $w=0,$ $C=-\frac{1}{\pi}\int_{\mathbb{C}}\frac{g(u)}{\overline{u}}da(u),$ and $\lambda\neq 0.$ For $N\geq 1$ \begin{equation} \widehat{h_{0}}^{N}(\lambda) = \frac{C^{N}}{\lambda^{N}} + \frac{N}{\lambda^{N}}\frac{1}{\pi}\int_{\mathbb{C}}u^{N}\left(\frac{1}{\pi}\int_{\mathbb{C}}\frac{h_{0}(v)}{v-u}da(v)\right)^{N-1}\frac{h_{0}(u)}{u-\lambda}da(u).
\end{equation} As a consequence,  in the sense of distributions, on $\mathbb{C}\backslash \lbrace 0\rbrace,$ \begin{equation}\label{power2} -\overline{\partial}\left( \widehat{h_{0}}^{N}\right) = N\widehat{h_{0}}^{N-1}h_{0}.
\end{equation}
\end {prop}

\begin{remark} We remark that the function $ f_{N}(u)=u^{N}\left(\frac{1}{\pi}\int_{\mathbb{C}}\frac{h_{0}(v)}{v-u}da(v)\right)^{N-1}$ appearing in this last integral extends to be continuous on $\mathbb{C}.$\end{remark}

On $\mathbb{C}\backslash \lbrace 0\rbrace,$ the series \[E(\lambda,0) = \sum_{n=0}^{\infty}\frac{(-1)^n}{n!}\widehat{h_0}^{n}(\lambda)\] converges in the sense of distributions. Consequently, using (\ref{power2}), on $\mathbb{C}\backslash \lbrace 0\rbrace,$ we have the distributional identity
\begin{equation}\begin{gathered}-\overline{\partial}E(\lambda, 0) = \sum_{n=0}^{\infty}\frac{(-1)^n}{n!}(-\overline{\partial})(\widehat{h_0}^{n})(\lambda) =\\ -\sum_{n=1}^{\infty}\frac{(-1)^{n-1}}{(n-1)!}(\widehat{h_0}^{n-1})h_{0}(\lambda)=\\ \frac{E(\lambda, 0)}{\overline{\lambda}}g(\lambda ).\end{gathered}\end{equation}

Thus, for fixed $w,$ we have established the distributional identity (\ref{dbarE}) which we emphasize is in the sense of distributions on $\mathbb{C}\backslash \lbrace w\rbrace.$  The distribution \[ \frac{E(u, w)}{\overline{u-w}}g(u)\] now considered as a locally integrable distribution on $\mathbb{C}$ differs from the distribution \[-\overline{\partial}[1-E(u,w)],\] on $\mathbb{C},$ by a first-order distribution supported on $\lbrace w\rbrace.$ Consequently in the sense of distributions on $\mathbb{C}$ \begin{equation}\label{firstorder}  -\overline{\partial}[1-E(u,w)] - \frac{E(u, w)}{\overline{u-w}}g(u) =\alpha\delta_{w} + \beta\partial\delta_{w}+ \gamma\bar{\partial}\delta_{w},\end{equation} for some constants $\alpha,\beta,\gamma.$ Let $S_{w}$ be the locally integrable distribution \begin{equation} \label{rightside} S_{w}(u) =1- E(u,w) - \frac{1}{\pi}\int_{\mathbb{C}}\frac{E(u,w)}{\overline{u-w}}\frac{g(u)}{u-\lambda} da(u) \end{equation} on $\mathbb{C}.$ Then equation (\ref{firstorder}) can be written in the form \[ -\overline{\partial} S_{w} = \alpha\delta_{w} + \beta\partial\delta_{w}+ \gamma\bar{\partial}\delta_{w}.\] We remark that when $w$ is a Lebesgue point for the function $g,$ then the estimate (\ref{lipord}) of Theorem \ref{main} implies that  $\frac{E(\lambda, w)}{\overline{\lambda -w}}g(\lambda)$ is in $L^{2+\sigma}(\mathbb{C})$ for some $\sigma >0.$  As indicated in Remark \ref{continuity} with $w$ fixed and$\frac{1}{\pi}\int_{\mathbb{C}} \frac{g(u)}{\vert u-w\vert^2} da(u )<\infty,$  the integral on the right side in \ref{rightside} is continuous. Consequently, in these cases, the distribution $S_{w}$ is a continuous function. It is an exercise in distribution theory to show that if $S$ is continuous and $-\overline{\partial} S = \alpha\delta_{w} + \beta\partial\delta_{w}+ \gamma\bar{\partial}\delta_{w},$ then $\alpha=\beta=\gamma=0$ and, therefore, $S$ is an entire function. Since $S_{w}$ vanishes at infinity, S must be zero. We have therefore established the following:

\begin{thm} \label{rep} Let $w$ be in $\bold{L}_{g}$  or satisfy $\frac{1}{\pi}\int_{\mathbb{C}} \frac{g(u)}{\vert u-w\vert^2} da(u )<\infty .$ Then (\ref{keyidentity}) holds for all $\lambda$ in $\mathbb{C}.$\end{thm}

\begin{remark} It would be interesting to know the extent to which (\ref{keyidentity}) holds and, in particular, whether the integral in this equation always converges at $w.$ 

We also remark that when $f$ is a compactly supported function that belongs to $L^{q}$ for $q >2,$ then $\hat{f}$ satisfies the Lipschitz condition $|\hat{f} (\lambda) - \hat{f} (w)| \leq K|\lambda - w|^{1-\frac{2}{q}}.$ For a proof of this last statement, see \cite{Brennan}. The results on the function $E$ in Theorem \ref{rep} and Theorem \ref{main} fall inline with this result. \end{remark}

\section{Operators with one-dimensional self-commutator}
As described in the introduction, there is a close connection between the function $E_{g}$ and operators with one dimensional self-commutator. In this section, we will describe some of these connections and derive a few consequences of the discussion of $E_{g}$ from the preceding section. Let $T$ be a bounded operator on the Hilbert space $\mathcal H$ satisfying $T^*T-TT^* = \varphi\otimes\varphi,$ where $\varphi$ is an element of $\mathcal H.$ It will always be assumed that the operator $T$ is irreducible, equivalently, there no non-trivial subspaces of $\mathcal H$ reducing $T$ where the restriction is a normal operator. Note that we have elected to assume $T$ is hyponormal, that is, the self-commutator $[T^*,T]=T^*T-TT^*$ is non-negative. The spatial behaviors of the hyponormal operator $T$ and its cohyponormal adjoint $T^*$ are quite distinct. As noted in the introduction,  for $\lambda$ in $\mathbb{C},$ there is a unique solution of the equation $T_{\lambda}^{*}x=\varphi$ orthogonal to the kernel of $T_{\lambda}^{*} = (T-\lambda)^{*},$ which will be denoted $T_{\lambda}^{*-1}\varphi.$  This follows easily from the range inclusion theorem of Douglas \cite{Douglas} when one notes that for all $\lambda\in\mathbb{C}$ one has $T_{\lambda}^*T_{\lambda} - T_{\lambda}T_{\lambda}^* =\varphi\otimes\varphi$ and, consequently, $T_{\lambda}^*T_{\lambda}\geq\varphi\otimes\varphi.$  The $\mathcal H$-valued function $T_{\lambda}^{*-1}\varphi$  defined for all $\lambda \in \mathbb{C}$ is called the global-local resolvent associated with the operator $T^*.$ Using the result of Douglas mentioned above one can also see that for all $\lambda\in\mathbb{C}$ there is a contraction operator $K(\lambda)$ satisfying $T_{\lambda}^* = K(\lambda) T_{\lambda},$ where $T_{\lambda} = T-\lambda.$ The contraction operator $K(\lambda)$ is unique if one requires it to be zero on the orthogonal complement of the range of $T_{\lambda}.$ The following identity, specialized here to the case of rank-one self-commutators, was first established for general hyponormal operators in \cite{Wadhwa} \[ I=T_{\lambda}^{*-1}\varphi\otimes T_{\lambda}^{*-1}\varphi + K(\lambda) K^{*}(\lambda) + P_{\lambda},\  \lambda\in\mathbb{C},\]  where $P_{\lambda}$ denotes the orthogonal projection onto the (at most one-dimensional) kernel of $T^*_{\lambda}.$ There is no equivalent of the global-local resolvent for the operator $T.$ To see this, we recall the following:
\begin{prop} \label{Kinvert} Let $T$ be an irreducible operator on the Hilbert space $\mathcal H$ with one-dimensional self-commutator $T^*T-TT^* = \varphi\otimes\varphi.$  Fix $\lambda\in\mathbb{C}.$ Then the following are equivalent: \begin{enumerate} \item[(I)] There is a solution of the equation $(T-\lambda)x=\varphi$ \item[(ii)]  The operator $K(\lambda)$ is invertible \item[(III)] $\| T_{\lambda}^{*-1}\varphi \| < 1$ \item[(IV)] $\int_{\mathbb{C}} \frac{g(u)}{\vert u-w\vert^2} da(u )<\infty,$ where $g$ is the principal function  of  $T.$\end{enumerate}\end{prop}

Let $\rho_{T}(\varphi)$ be the set of $\lambda\in\mathbb{C}$ such that there is a solution of the equation $(T-\lambda)x=\varphi.$ The last condition in the above proposition shows that condition ($I$) cannot hold at Lebesgue points of $g.$ This implies the result of Putnam \cite{Putnam} which establishes that the interior of $\rho_{T}(\varphi)\cap\sigma (T),$ where $\sigma (T)$ denotes the spectrum of $T,$  is empty and points out the significant difference between the local resolvents  $T_{\lambda}^{*-1}\varphi$ and $T_{\lambda}^{-1}\varphi. $

We will continue to use the notation $\bold{L}_g$ for the set of points of positive Lebesgue density of a bounded measurable function $g.$ It develops that the local resolvent function is locally Lipschitz on $\bold{L}_g.$ This is the content of the following: 

\begin{prop} Let $g$ be the the principal function associated with the operator $T$  having one-dimensional self-commutator $T^*T-TT^* = \varphi\otimes\varphi.$ Let $w$ be in the set $\bold{L}_g$ with $0<\gamma =\lim_{R\to 0}\frac{1}{\pi R^{2}} \int_{D(w,\delta)}g(u)da(u).$ Then there is a disc $D(w, \delta)$ such that \begin{equation} \| T_{\lambda}^{*-1}\varphi - T_{w}^{*-1}\varphi\|\leq K|\lambda - w|^{\gamma -\varepsilon}, \end{equation} for $\lambda\in D(w,\delta)\cap\bold{L}_{g},$ where $K$ is a constant,  and $\epsilon < \gamma.$
\end{prop}

\begin{proof} It follows from Theorem \ref{main} that there is a disc $D(w, \delta)$ so that for $\lambda\in D(w,\delta)$ one has \begin{equation}\label{est} |E_{g}(\lambda,w)| \leq K|\lambda - w|^{\gamma -\varepsilon}.\end{equation} Since, $\| T_{\lambda}^{*-1}\varphi\| = 1$ for $\lambda\in\bold{L}_{g}$ \begin{equation}\begin{gathered} \| T_{\lambda}^{*-1}\varphi - T_{w}^{*-1}\varphi\|^2=\\\|T_{\lambda}^{*-1}\varphi  \|^2 -<T_{\lambda}^{*-1}\varphi, T_{w}^{*-1}\varphi > - <T_{w}^{*-1}\varphi, T_{\lambda}^{*-1}\varphi >
+ \|T_{w}^{*-1}\varphi\|^2= \\1 - <T_{\lambda}^{*-1}\varphi, T_{w}^{*-1}\varphi > + 1 - <T_{w}^{*-1}\varphi, T_{\lambda}^{*-1}\varphi > =\\ \overline{E_{g}(\lambda,w)} +E_{g} (\lambda,w) \leq 2|E_{g}(\lambda, w)|,
\end{gathered}\end{equation} where we have made use of Theorem \ref{KFC}. The result follows from (\ref{est}).
\begin{remark} The result in this last proposition will not be true at points $w$ in the spectrum $\sigma (T)$ where any of the conditions in Proposition \ref{Kinvert} hold. It is easy to construct an example of an operator $T$ with this property, so that, $w\in\sigma (T)$ and $\|T_{w}^{*-1}\varphi\| <1.$ By Putnam's result mentioned above, in every neighborhood of $w,$ there will be points with $\| T_{\lambda}^{*-1}\varphi\| =1.$ Thus the conclusion of this last proposition cannot hold at the point $w.$\end{remark}
\end{proof}
\subsection{Integral representations using the global-local resolvent}
The connection between the global local resolvent and the principal function can be seen in the following result from \cite{ClanceyIU} 
\begin{thm} \label{CL} Let $T$ be an operator with one-dimensional self-commutator $T^*T-TT^* = \varphi\otimes\varphi$ and $g$ the associated principal function. For $r=r(u)$ a rational function with poles off the spectrum of $T$  and $\lambda\in\mathbb{C}$ \begin{equation} \label{Cauchy} (r(T)\varphi, T_{\lambda}^{*-1}\varphi) = \frac{1}{\pi}\int_{\mathbb{C}} r(u) \frac{g(u)}{u-\lambda }da(u). \end{equation} As a consequence, in the sense of distributions, \begin{equation}\label{dbar} -\bar{\partial} (\varphi, T_{\lambda}^{*-1}\varphi) = g.\end{equation} \end{thm}

We record the following analogue of this last result for the operator $T^{*}.$

\begin{thm} \label{CLstar} Let $T$ be an operator with one-dimensional self-commutator $T^*T-TT^* = \varphi\otimes\varphi$ and $g$ the associated principal function. Assume $\lambda\in\bold{L}_{g}.$  Let $p=p(\bar{u})$ be a polynomial in the variable $\bar{u},$ then
\begin{equation}\label{Cauchystar} (p(T^{*})\varphi, T_{\lambda}^{*-1}\varphi) =  \frac{1}{\pi}\int_{\mathbb{C}} p(\bar{u}) \frac{ E_{g} (\lambda,  u)  }{u-\lambda } g(u)da(u). \end{equation}\end{thm}
\begin{proof} Since $\lambda\in\bold{L}_{g},$ then by Theorems \ref{KFC} and \ref{rep} for all $z$ \begin{equation} \label{twovariable}< T_{\lambda}^{*-1}\varphi, T_{z}^{*-1}\varphi>  = 1 - E(z,\lambda) = \frac{1}{\pi}\int_{\mathbb{C}}\frac{E(u,\lambda)}{\overline{u-\lambda}} g(u)\frac{da(u)}{u-z}.\end{equation} Equating powers of $z$ at infinity one obtains \begin{equation} < T_{\lambda}^{*-1}\varphi, T^{*k}\varphi>  =\frac{1}{\pi}\int_{\mathbb{C}}u^{k}\frac{E(u,\lambda)}{\overline{u-\lambda}} g(u)da(u),
\end{equation} for $k=0, 1,\cdots .$ The result follows by taking complex conjugates in this last identity.

\end{proof}

\begin{remark} The integral kernel \[ \mathcal{T} (\lambda, u) =\frac{E_{g}(\lambda, u)}{u-\lambda}g(u),\] appearing in Theorem \ref{CLstar} $a.e.\ gda$ satisfies $\vert \mathcal{T} (\lambda, u)\vert \leq \vert u-\lambda\vert^{-\sigma}$ near $\lambda,$ where $0<\sigma <1.$ In particular, has the advantage that  $\mathcal{T} (\lambda, \cdot)$ is in $L^2(\mathbb{C}).$

It is known that the closed span of $T_{\lambda}^{*-1}\varphi,\ \lambda\in\sigma(T),$ is $\mathcal{H}.$ Since the closure of $\bold{L}_{g}$ equals $\sigma (T),$ it follows that the closed span of $T_{\lambda}^{*-1}\varphi,\ \lambda\in\bold{L}_{g},$ also is $\mathcal{H}.$ \end{remark}
We conclude this subsection with a few examples of applications of Theorem \ref{CLstar}..
A test-function model for the operator $T^{*}$ was constructed in \cite{putdist} (see also \cite[p. 151, p.261]{putinarmartin} ) using the map \[\eta\in\mathcal{D}(\mathbb{C})\ \rightarrow\ U(\eta) = \frac{1}{\pi}\int\partial\eta (\lambda) T_{\lambda}^{*-1}\varphi da(\lambda ),\]  so that $U(\bar{z}\eta) = T^{*}U(\eta).$ This test function model is dual to the distributional model described in \cite{ClanceyJOT}, where the map $V:\mathcal{H} \rightarrow \mathcal{E}^{'}(\mathbb{C})$ defined by $V(f)=-\bar{\partial} <f, T_{\lambda}^{*-1}\varphi > $ was studied.  It is easily verified that  for $\eta$ a test function and $f\in\mathcal{H}$
\[ <U(\eta), f> = \frac{1}{\pi}<\eta, \overline{V(f)}>.\]

Note that using (\ref{twovariable}) for $w\in\bold{L}_{g}$ \begin{equation}\begin{gathered} < T_{w}^{*-1}\varphi, U(\eta ) > = \frac{1}{\pi}\int_{\mathbb{C}}\overline{\partial}\overline{\eta} < T_{w}^{*-1}\varphi, T_{\lambda}^{*-1}\varphi > da(\lambda) = \\ \frac{1}{\pi}\int_{\mathbb{C}}\bar{\partial} \bar{\eta}(\lambda)\left(\frac{1}{\pi}\int_{\mathbb{C}}\frac{E(u,w)}{\overline{u-w}} \frac{g(u)}{u-\lambda} da(u)\right) da(\lambda)=\\\frac{1}{\pi}\int_{\mathbb{C}}\ \bar{\eta}(\lambda)\frac{E(\lambda,w)}{\overline{\lambda-w}} g(\lambda)da(\lambda).
 \end{gathered} \end{equation} Equivalently, \begin{equation}  <U(\eta ), T_{w}^{*-1}\varphi > = \frac{1}{\pi}\int_{\mathbb{C}}\eta(\lambda)\frac{E(w,\lambda )}{\lambda-w} g(\lambda)da(\lambda).
 \end{equation}\

\begin{example} In some cases, when combined with Theorem \ref{CLstar}, this last equation allows one to to identify the vector  $U(\eta) .$ For example, if $\eta (u) = \bar{u}^{k}, \  k= 0, 1, 2, \cdots$ on the set where $g$ is non-zero, then $U(\eta ) = T^{*k}\varphi.$  It is noted that this result can also be obtained directly from the definition of $U(\eta).$ \end{example}
\begin{example} For simplicity, suppose the essential support $\bold{G}$ of the principal function $g$ is contained in the open unit disc $\mathbb{D}.$ Let \[ \Phi_{1} = \frac{1}{\pi}\int_{\mathbb{D}}T_{\lambda}^{*-1}\varphi da(\lambda ) \ \text{and}\  \Phi_{g} = \frac{1}{\pi}\int_{\mathbb{C}}T_{\lambda}^{*-1}\varphi g(\lambda) da(\lambda ).\] For $w\in \bold{L}_{g},$ using (\ref{twovariable}) one computes 
\[<T_{w}^{*-1}\varphi, \Phi_{1} +\Phi{g} > = <T_{w}^{*-1}\varphi, T\varphi>.\] This yields \[ T\varphi = \Phi_{1} + \Phi_{g} = \frac{1}{\pi}\int_{\mathbb{D}} (1+g) T_{\lambda}^{*-1}\varphi da(\lambda), \] which gives a concrete representation of $T\varphi$ in terms of the dense family $\lbrace T_{\lambda}^{*-1}\varphi : \lambda\in\sigma (T)\rbrace.$  As is often the case, the unilateral shift provides an illuminating version of this last identity.This last identity can be viewed as a realization of the formula for $T\varphi$ given in the test function model in \cite[p. 261]{putinarmartin}. That is \[ T\varphi = \frac{1}{\pi}\int\partial\left(\eta -\overline{\widehat{g\bar{\eta}}}\right) T_{\lambda}^{*-1}\varphi,\]  where the test function $\eta$ satisfies $n(\lambda) = \lambda$ on the support of $g.$ 
\end{example}
\subsection{Non-cyclic vectors}
Let $T$ be an operator with one-dimensional self-commutator $T^*T-TT^* = \varphi\otimes\varphi$ and $T_{\lambda}^{*-1}\varphi\ \lambda\in\mathbb{C}$ the corresponding global-local resolvent. Given a compactly supported planar measure $\mu,$ one can define the vector \begin{equation}\label{vector}\phi_{\mu} =\int_{\mathbb{C}}T_{\lambda}^{*-1}\varphi d\mu\end{equation} as a weak integral, that is, for $f\in\mathcal {H}$
 \[ <\phi_{\mu}, f > = \int_{\mathbb{C}}<T_{\lambda}^{*-1}\varphi, f > d\mu (\lambda). \] In  an extremely formal sense $\phi_{\mu} = -(\widehat{\overline{\mu}}(T))^{*}\varphi.$ For example, if $\mu = \delta_{w},$ with $\mu\notin\sigma (T),$ one has $-\widehat{\overline{\mu}}(\lambda) = \frac{1}{\pi (\lambda -w)} $ and $\phi_{\mu} = T_{w}^{*-1}\varphi.$ 
 
It follows from Theorem \ref{CL} that for $r$ a rational function with poles off the spectrum of $T$ we have \[ <r(T)\varphi, \phi_{\mu} > = -\int_{\mathbb{C}} r(u)\widehat{\overline{\mu}} (u) g(u) da(u).\] We remark on the connection between this last identity and results concerning rational approximation in \cite{Thomson} and more recently \cite{yang}. For $X$ a compact set in the plane, let $P(X)$ respectively, $R(X)$  be the closure in the space $C(X)$ of continuous function on $X$ of the polynomials, respectively, the rational functions with poles off $X.$ It was shown in \cite{Thomson} when $X$ is nowhere dense, then the closure of the module $\bar{z} P(X) + R(X)$  is $C(X)$ if and only if $R(X)=C(X).$ Here $z$ is the function $z(u)=u.$ This is in contrast to the result in \cite{trent} that establishes when X is a compact nowhere dense set, then the closure of $\bar{z}R(X)+R(X)$ is $C(X).$ 

We are interested here in the case where the characteristic function $\mathbbm{1}_{X}$ ``is" the principal function of an operator with one-dimensional self-commutator. Since the principal function of an operator with one-dimensional self commutator is only determined up to sets of measure zero, this has to be properly interpreted. If $g$ is the principal function associated with $T,$ then $\sigma (T)$ is the essential closure of the set $\lbrace u:g(u)\neq 0 \rbrace.$ Consequently, we only consider the class of  closed nowhere dense sets of positive measure that are essentially closed, i.e., equal their essential closure. For such a set $X$ there is a unique associated irreducible operator $T_{X}$ with one-dimensional self-commutator having principal function $\mathbbm{1}_{X}.$ Moreover, different such sets $X$ correspond to different operators $T_{X}$ and $\sigma (T_{X}) =X.$ It should also be noted that for a compact set $X$ the essential closure of $X$ is a closed subset of $X$ that differs from $X$ by a set of planar measure zero.  Moreover, if $R(X)\neq C(X),$ the same is true for its essential closure \cite{rubel}.

Based on the result of Thomson \cite{Thomson}, as noted in \cite{ClanceyIU}, one can establish the following:
\begin{prop} Suppose $T$ is an operator with one-dimensional self-commutator $T^*T-TT^* = \varphi\otimes\varphi$ associated with the principal function $\mathbbm{1}_{X},$ where $X$ is a compact  essentially closed nowhere dense set of positive measure. If the closure of $\bar{z} P(X) + R(X)$ is not equal to $ C(X), $ equivalently, the closure of $\bar{z} P(\sigma (T)) + R(\sigma (T))$ is not equal to $C(\sigma (T)),$ then the vector $\varphi$ is not cyclic for the operator $T.$ \end{prop}
\begin{proof} Let $\mu$ be a non-zero measure on $X$ that annihilates $\bar{z} P(X)+R(X).$ Let $\phi_{\bar{\mu}}$ be given by (\ref{vector}) with $\bar{\mu}$ replacing $\mu.$ Then for $r$ a rational function with poles off $X$ the last equation results in the identity
 \begin{equation} <r(T)\varphi, \phi_{\bar{\mu}} > = -\int_{X} r(u)\hat{\mu} (u) da(u)= \int_{X}\bar{u}r(u)d\mu (u).\end{equation}
By the result of \cite{trent} the closure of $\bar{z}R(X) + R(X)$ is $C(X)$ and therefore for some $r\neq 0$ the right side of this last equation is non-zero. This implies $\phi_{\bar{\mu}}$ is non-zero. Since $\int_{X}\bar{z}pd\mu = 0$ for all polynomials $p$ it follows that $\varphi$ is not (polynomially) cyclic for the operator $T.$\end{proof}
\begin{remark} In the context of the last result, a natural example to consider is that of a Swiss cheese $X,$ that is, where $X$ is a closed nowhere dense set of positive planar measure obtained by removing a collection of open discs $D(w_{n}, r_{n}),\ n=1,2\cdots $ with disjoint closures from the closed unit disc. If one assumes $\Sigma_{1}^{\infty} r_{n} <\infty,$ then $R(X)\neq C(X).$ Let $T$ be the irreducible operator with one-dimensional self-commutator associated with the principal function $\mathbbm{1}_{X},$  where $X$ is a Swiss cheese with $\Sigma_{1}^{\infty} r_{n} <\infty.$ It is known that the $\varphi$ is rationally cyclic for both $T$ and $T^{*},$ see \cite{ClanceyJOT}. The result above shows $\varphi$ is not (polynomially) cyclic for $T.$

In a fundamental paper, Brown \cite{brown} established the existence of invariant subspaces for a hyponormal operators $H$ whenever there is a closed disc $D$ such that $R(D\cap\sigma (H))\neq C(D\cap\sigma (H)).$  Thus the result in the last proposition does not advance the theory of invariant subspaces. However, the fact that $\varphi$ is not cyclic for $T$ is new, albeit depending on the deep result of Thomson \cite{Thomson}. 

The result in the last proposition can be extended to the case, again assuming $X$ is nowhere dense, where for some closed disc $D$ the closure of $\bar{z} P(X\cap D) + R(X\cap D)$ is not equal to $ C(X\cap D).$ To see this, note if $\mu$ is a non-zero measure on $X\cap D$ annihilating $\bar{z} P(X\cap D) + R(X\cap D),$  then  one has \[<r(T)\varphi, \phi_{\bar{\mu}} > =  \int_{X\cap D}\bar{u}r(u)d\mu (u)\]  for $r$ a rational function with poles off $X$ and the last integral will be zero for $r = p$ a polynomial. In order to see that $\phi_{\bar{\mu}}$ is not zero, note that the last integral will be non-zero for some rational function $r$ with poles off $X\cap D.$ If one does a partial fractions decomposition of $r$ the part of this decomposition with poles in $X\backslash(X\cap D)$ can be approximated by a polynomial on $X\cap D.$ In this way, $r$ can be replaced by a rational function $r_{0}$ with poles off $X$ where $\int_{X\cap D}\bar{u}r_{0}(u)d\mu (u)\neq 0.$

It would be interesting to see if the results in \cite{trent1}, \cite{Thomson} and \cite{yang} appropriately extend so that the above techniques can be used to establish that the vector $\varphi$ is not cyclic under the conditions that $\sigma (T)$ is nowhere dense and there is a closed disc $D$ such that $R(D\cap\sigma (H))\neq C(D\cap\sigma (H)).$
\end{remark}

\subsection{Some definite integral values computed using $E$ and the unilateral shift}
In the case, where the operator  $T$ is the unilateral shift $Uf(z)=zf(z)$ acting on the Hardy space $H^{2}$ consisting of analytic functions $f$ on the open unit disc $\bold{D}$ with norm \[ \| f\| = \left(\lim_{r\to 1}\frac{1}{2\pi}\int_{0}^{2\pi} |f(re^{i\theta}|^2d\theta\right)^{\frac{1}{2}}, \] the principal function $g_{T}$ is the characteristic function of the disc $\bold{D}.$ It is easy to verify that for $|\lambda |<1, $
\begin{equation}\label{inside}U_{\lambda}^{*-1} \bold{1} (z) =\frac{z-\lambda}{1-\bar{\lambda}z} \end{equation} and for $|\lambda |\geq 1,$ \begin{equation}\label{outside}U_{\lambda}^{*-1} \bold{1} (z) = -\frac{\bold{1}}{\bar{\lambda}}, \end{equation} where we are using the notation $\bold{1}$ for the constant function $\bold{1}(z) = 1$ that appears in the self-commutator $U^{*}U - UU^{*} = \bold{1}\otimes\bold{1}.$ A straightforward computation can be used to verify \begin{equation}\label{shift}1- (U_{w}^{*-1}\bold{1},U_{\lambda}^{*-1}\bold{1})= E_{\bold{D}}(\lambda, w ) =  \exp-\frac{1}{\pi}{\int_{\bold{D}} \frac{u-w}{u-\lambda }\frac{1}{\vert u-w\vert^2}} da(u ),\end{equation} where $E_{\bold{D}}$ is given by (\ref{unitdisc}). 

Using this last formula one can directly verify the integral formulas (\ref{real}) and (\ref{Imaginary}). 
For example, consider the formula (\ref{real}) where $\alpha <1.$ The map $u = (\lambda -c_{\alpha})z +c_{\alpha}$ maps $\bold{D}$ onto $D_{\lambda, \alpha}$ sending $-1$ to $\lambda_{\alpha} = w +\alpha(\lambda - w)$ and $1$ to $\lambda.$ This change of variables results in the equality \begin{equation} \label{change}-\frac{1}{\pi}\int_{D_{\lambda,\alpha}} \frac{u-w}{u-\lambda }\frac{1}{\vert u-w\vert^2} da(u ) =-\frac{1}{\pi}\int_{\bold{D}}\frac{z - s_{\alpha}}{z-1}\frac{1}{|z-s_{\alpha}|^{2}} da(z),
\end{equation} where $s_{\alpha} = \frac{1+\alpha}{\alpha -1}$ is the image of $w$ under the inverse map $z=z(u).$ The right side of this last equation is recognized as the exponent in the following special case of equation (\ref{shift}):
\begin{equation} 1- (U_{s_{\alpha}}^{*-1}\bold{1},U_{1}^{*-1}\bold{1})= E_{\bold{D}}(s_{\alpha}, 1 ) =  \exp-\frac{1}{\pi}{\int_{\bold{D}} \frac{u-s_{\alpha}}{u-1 }\frac{1}{\vert u-s_{\alpha}\vert^2}} da(u ).
\end{equation}  Depending on whether $s_{\alpha}$ is inside the open unit disc ($\alpha <0$) or outside the open unit disc ($0\leq \alpha <1$) one uses (\ref{inside}) or (\ref{outside}) to compute the left side of this last equality. For example, in the case where $s_{\alpha}$ is inside the open unit disc \[1- (U_{s_{\alpha}}^{*-1}\bold{1},U_{1}^{*-1}\bold{1}) = \frac{2}{1+|\alpha |}.\] Combining this last identity with (\ref{change}) equation (\ref{real}) follows. Similar arguments using the unilateral shift can be used to obtain the other instances of (\ref{real}) and (\ref{Imaginary}).

\bibliographystyle{plain}
\bibliography{Ework}

Department of Mathematics

University of Georgia

Athens, GA 

email: kclancey@uga.edu

\end{document}